\documentclass{amsart}

\usepackage{amssymb} 
\usepackage{epstopdf} 
\usepackage{amsmath} 
\usepackage{amsthm} 
\usepackage{mathrsfs} 
 
\newtheorem{theorem}{Theorem}[section]

\newtheorem{corollary}[theorem]{Corollary} 
\newtheorem{proposition}[theorem]{Proposition}

\newtheorem{condition}[theorem]{Condition} 

\DeclareMathOperator{\rk}{rk} 
\DeclareMathOperator{\gk}{\mathcal{GK}} 
\DeclareMathOperator{\s}{\ast} 
\DeclareMathOperator{\Hom}{Hom}  

\begin{document} 

\title{A note on the simple modules over McConnell--Pettit algebras} 
\author{Ashish Gupta}
\address{
Harish-Chandra Research Institute,\\
Chhatnag Rd., Jhunsi,\\
Allahabad. 211019\\
India.}

\email{agupta@hri.res.in}



\date{June 17, 2011}

\begin{abstract}
We consider simple modules over the McConnell--Pettit algebras.
We show that both induction and contraction yield 
simple modules for the extremes of the global dimension. 
\end{abstract}

\maketitle 

\section{Introduction}

We recall that a quantum polynomial algebra over a field $F$ is defined as the associative algebra generated over $F$ by 
$y_1, \cdots, y_n$, $n > 1$, subject only to the relations:
\begin{equation}
\label{0.1}
y_iy_j = q_{ij}y_jy_i,
\end{equation} 
where $1 \le i,j \le n$ and $q_{ij} \in F$ are nonzero scalars. Note that 
\[ q_{ii} = 1 = q_{ij}q_{ji}.  \]  
These algebras play an important role in noncommutative geometry (see \cite{MA}).
We denote by $Q$ the matrix $(q_{ij})$ of \emph{multiparameters}.

Localization at the monoid generated by $\{y_1, \cdots, y_n\}$ yields $\Lambda(F, Q)$, the so called \emph{Multiplicative analogue of the Weyl algebra}. Clearly, $\Lambda(F, Q)$ has $y_1^{\pm 1}, \cdots, y_n^{\pm 1}$ for generators and (\ref{0.1}) for the  defining relations. 
It is also known as \emph{quantum Laurent polynomial algebra}, \emph{McConnell--Pettit algebra} and $\mathbb Z^n$-\emph{quantum torus}.  

We recall the structure of a twisted group algebra $F \s A$  (e.g., \cite{PA}) of a finitely generated torsion-free abelian group $A$ over $F$. 
This is an $A$-graded algebra $\oplus_{a \in A} F \bar a$, where multiplication is a ``twisted" version of the group multiplication:    
\[ \bar {a}_1 \bar {a}_2 = \lambda(a_1, a_2)\overline {a_1a_2}, \]
where $a_1, a_2 \in A$ and $\lambda : A \times A \rightarrow F^*$ is a $2$-cocycle: 
\begin{equation}
\label{cocy}
 \lambda(a_1, a_2)\lambda(a_1a_2, a_3) = \lambda(a_2, a_3)\lambda(a_1, a_2a_3), \ \ \ \ \  a_1, a_2, a_3 \in A. 
\end{equation}
The $\Lambda(Q,F)$ are precisely the twisted group algebras $F \s A$.
By definition, $\alpha \in F \s A$ may be presented as $\alpha = \sum_{a \in A}\mu_a \bar a$, where 
$\mu_a \in F$ is nonzero only for a finite (possibly empty) subset of $A$ known as the \emph{support} of $\alpha$ in $A$. 
For a subgroup $B < A$,  the subalgebra of elements $\alpha$ with support contained in $B$ is twisted group algebra $F \s B$ of $B$ over $F$.    
If $A/B$ is infinite cyclic, say $A/B = \langle tB \rangle$ for $t \in A$, then 
$F \s A$ is a skew-Laurent extension 
\[ F \s A = (F \s B)[\bar t^{\pm 1}, \sigma], \]   
where $\sigma$ is the automorphism of $F \s B$ given by
\[ \sigma(\beta) = \bar t \beta \bar t^{-1} \] 
for all $\beta \in F \s B$.  

\section{Simple modules}

Our focus shall be on the simple $F \s A$-modules. 
We can associate with each finitely generated $F \s A$-module $M$ its Gelfand--Kirillov dimension (GK dimension) which is  a measure of the growth of $M$. To give some idea of this dimension, let $a_1, \cdots, a_n$ be a basis of $A$ and $V_0$ be the subspace of $F \s A$:    
\[ V_0 = F + \sum F\bar {a}_i + \sum F\bar {a}_i^{-1} \] 
Let $W_0$ be a finite dimensional subspace of $M$ such that $W_0(F \s A)= M$. 
Define $f : \mathbb N \rightarrow \mathbb N$ by
$f(m) = \dim_F(W_m)$, where $W_m = W_0V_0^m$.
Then $\gk(M) = \lim \sup (\log f(n)/\log n)$.
For details on the GK dimension, we refer the interested reader to \cite{KL} or \cite[Chapter 8]{MR}. 
 
A dimension for modules over crossed products was introduced and studied by C.J.B. Brookes and J.R.J. Groves in \cite{BG}. It was shown to coincide with the GK dimension. The following characterization of GK dimension for 
$F \s A$-modules thus follows from \cite{BG}:
 
\begin{proposition}[\cite{BG}]
\label{BG_dim} 
Let $M$ be a finitely generated $F \s A$-module. Then $\gk(M)$ equals the supremum of the ranks of subgroups 
$B$ of $A$ such that $M$ is not torsion as $F \s B$-module. Furthermore, let $a_1, \cdots, a_n$ 
freely generate $A$ and $\mathcal F$ denote the family of subgroups of $A$: 
\[ \mathcal F = \{ \langle X \rangle : X \subset \{a_1, \cdots, a_n \} \} \]
with the convention that $\langle \emptyset \rangle = \langle 1 \rangle $. 
Then $\gk(M)$ is the supremum of the ranks of subgroups $B$ 
in $\mathcal F$ such that $M$ is not $F \s B$-torsion.  
\end{proposition}

A lower bound for the GK dimension of a nontrivial finitely generated $F \s A$-module has been obtained by Brookes in \cite{BR}. It is natural to expect simple modules among modules having the least possible GK dimension.
Note that $A$ may contain subgroups $B$ so that $F \s B$ is commutative, that is, the cocycle $\lambda$ may be trivial when restricted to $B$. 

\begin{theorem}[Theorem 3 of \cite{BR}]
\label{B_T3}
If $F \s A$ has a finitely generated nonzero module $M$ with $\gk(M) = s$, then $A$ has a subgroup $B$ with corank $s$ such that $F \s B$ is commutative. 
\end{theorem}  

In this connection, the following theorem is shown in \cite{BR} which was initially a conjecture in \cite{MP}:

\begin{theorem}[Theorem A of \cite{BR}]
\label{MP_conj}
The global and Krull dimensions of $F \s A$ equal the supremum of the ranks of subgroups $B \le A$ so that $F \s B$ is commutative. 
\end{theorem}  

We let $\dim(F \s A)$ stand for either of these dimensions. As a corollary we have,

\begin{corollary}
\label{B_ine}
If $M$ is a nonzero finitely generated $F \s A$-module then  
\begin{equation} 
\gk(M) \ge \rk(A) - \dim(F \s A) 
\end{equation}
\end{corollary}

Thus the minimum possible GK dimension for a nontrivial finitely generated $F \s A$-module $M$ is $n - d$, where 
$d = \dim(F \s A)$. For $d = 1$, this was already shown in \cite[Theorem 6.2]{MP}. 

\begin{proposition}
\label{hol_mod}
Let $M$ be a finitely generated $F \s A$-module. If $\gk(M) = \rk(A) - \dim(F \s A)$ then $M$ has  finite length.
\end{proposition}

\begin{proof}
Let 
\begin{equation}
\label{des_ch} 
M = M_0 > M_1 > \cdots > M_i > M_{i + 1} > \cdots  
\end{equation}
be a (strictly) descending chain of submodules in $M$.
By Corollary \ref{B_ine}, $\gk(M_i/M_{i + 1}) = \gk(M)$ for $i \ge 0$. By \cite[Lemma 5.6 and Section 5.9]{MP}, 
the sequence (\ref{des_ch}) halts.
\end{proof}

We now fix some notation. For the remainder of this note, $B$ is always a subgroup such that $A/B$ is infinite cyclic.
We set $\mathcal B = F \s B$, $\mathcal A = F \s A$ and $n = \rk(A)$. As already observed in the introduction, $\mathcal A = \mathcal B[\bar t^{\pm 1}, \sigma]$, where 
$t$ generates $A$ modulo $B$. It is known (e.g. \cite{MP}) that 
$S := \mathcal B \setminus \{0\}$ is an Ore subset in $\mathcal A$. The corresponding ring of fractions is a skew-Laurent polynomial ring $\Sigma = \mathscr D[\bar t^{\pm 1}, \sigma]$, where $\mathscr D = \mathcal B S^{-1}$ is the quotient division ring of $\mathcal B$.  We have thus embedded $\mathcal A$ in a PID. 

\subsection{Induction from simple $\mathcal B$-modules}
 
Let $V$ be any simple $\mathcal B$-module. Then $W = V \otimes_{\mathcal B} \mathcal A$ is an $\mathcal A$-module which in general may not be simple or even artinian. It is, however, GK-\emph{critical} in the sense that if $W_1 < W$ is a nonzero submodule of $W$ then $\gk(W/W_1) < \gk(W)$. 
This is an immediate consequence of \cite[Lemma 2.4]{BG}.
It was shown in \cite[Theorem 6.1]{HM} that if 

\begin{condition}
\label{HM_con}
No simple $\mathcal A$-module has finite length as $\mathcal B$-module
\end{condition}  

holds then $W$ is simple. 
Since $\mathcal B$ is noetherian, the above condition is the same as demanding that if a simple $\mathcal A$-module is finitely generated as $\mathcal B$-module then it has an infinite strictly descending chain of $\mathcal B$-submodules.
Since $\mathcal B$ is not (right) artinian, it follows that Condition \ref{HM_con} is implied by the following 

\begin{condition}
\label{HM_str_con1}
If a simple $\mathcal A$-module $U$ is finitely generated as $\mathcal B$-module, then $U$ is not $\mathcal B$-torsion.    
\end{condition}

We shall now give some examples where Condition \ref{HM_str_con1} (and hence \ref{HM_con}) is satisfied. 
By Lemma \cite[Lemma 2.4]{BG}, $\gk(W) = \gk(V) + 1$ holds irrespective of whether $W$ remains simple or not.

\begin{proposition}
\label{dim_n-1_cs}
If $\mathcal B$ is commutative and $\mathcal A$ has center $F$ then Condition \ref{HM_str_con1} holds. 
\end{proposition}
\begin{proof}
Let $U$ be a simple $\mathcal A$-module finitely generated as $\mathcal B$-module.
Let $U'$ denote the $\mathcal B$-module structure of $U$. By \cite[Lemma 2.7]{BG}, $\gk(U') = \gk(U)$.
In view of Proposition \ref{BG_dim}, let $E < B$ be such that $B/E$ is torsion-free, $\rk(E) = \gk(U')$ and $U'$ is not 
$F \s E$-torsion. If $\rk(E) = \rk(B)$, it follows that $E = B$ and thus $U$ is not $F \s B$-torsion. 

Suppose that $\rk(E) < \rk(B)$. 
By \cite{GU2} or \cite[Lemma 4.10]{GU1}, there is a subgroup $E' < A$ such that $\rk(EE') = \rk(A)$, $E \cap E' = \langle 1 \rangle$ and $F \s E'$ is commutative. Then $E'\cap B > \langle 1 \rangle$ and $F \s (E' \cap  B)$ is central in $F \s (EE')$. Since $[A : EE'] < \infty$, it follows that $F \s A$ has center larger than $F$ contrary to the hypothesis.     
\end{proof}

\begin{proposition}
\label{dim_1_cs}
If $\dim(\mathcal A) = 1$, then Condition \ref{HM_str_con1} is satisfied.
\end{proposition}
\begin{proof}
Let $U$ be a simple $\mathcal A$-module that is finitely generated as $\mathcal B$-module. Let $U'$ be its 
$\mathcal B$ module structure. Then $\gk(U) = \gk(U')$ by \cite[Lemma 2.7]{BG}. If $\gk(U') = \rk(B)$ then $U'$ (and hence $U$) is not $\mathcal B$-torsion by Proposition \ref{BG_dim} and we are done.

Assume that $\gk(U') < \rk(B)$. Then $\dim(\mathcal A) > 1$ in view of Corollary \ref{B_ine} contrary to the hypothesis.  
\end{proof}

We have just seen that for $\dim(\mathcal A) = 1$ or $n - 1$ and $\mathscr Z(\mathcal A) = F$, where 
$\mathscr Z(\mathcal A)$ denotes the center of $\mathcal A$, modules induced from simple $\mathcal B$-modules 
remain simple.  We shall see in the next section that for these two cases modules obtained by contraction (as explained below) are also simple. It is possible to construct other examples satisfying Condition \ref{HM_str_con1} using \cite[Theorem 3.9]{MP}.

\subsection{Contraction of maximal right ideals of $\Sigma$}
Let $\mathfrak m$ be a maximal right ideal in $\Sigma$. Then $\mathfrak m$ is generated by an irreducible in $\Sigma$.
Moreover $S(\mathfrak m) = \mathcal A/\mathcal A \cap \mathfrak m$ is a cyclic $\mathcal A$-module that is torsion-free over 
$\mathcal B$.
Following \cite{AR}, an element $\alpha \in \mathcal A$ is  \emph{unitary with respect to} $\bar t$ if in the (unique) expresssion:
$\alpha = \sum_{i = p}^q \beta_i\bar t^i$, where $p \le q \in \mathbb Z$ and $\beta_i \in \mathcal B$, the terminal coefficients $\beta_p$ and $\beta_q$ are units. By \cite{AR}, for a right ideal $\mathcal I$ of $\mathcal A$, 
$\mathcal A/ \mathcal I$ is finitely generated as $\mathcal B$-module if and only if $\mathcal I$ contains an element unitary with respect to $\bar t$. 
It is shown in \cite[Lemma 3.4]{BVO} that $S(\mathfrak m)$ is simple if and only if $\Hom_{\mathcal A}(S(\mathfrak m), N) = 0$ whenever $N$ is a simple $\mathcal A$-module with $S$-torsion. Recall that we have defined $S$ as $S = \mathcal B \setminus \{0\}$. 

\begin{proposition}
\label{GKdim_ft}
With the above notation, $\gk(S(\mathfrak m)) = \rk(B)$. 
\end{proposition}

\begin{proof}
As noted above, $S(\mathfrak m)$ is finitely generated and torsion-free as $\mathcal B$-module. 
The proposition follows from \cite[Lemma 2.7]{BG} and Propostion \ref{BG_dim}. 
\end{proof}

\begin{proposition}
\label{ctr_max_idl}
If $\mathcal A \cap \mathfrak m$ has an element unitary with respect to $\bar t$ then $S(\mathfrak m)$ is GK-critical. 
\end{proposition}

\begin{proof}
Let $\mathcal I$ be a right ideal of $\mathcal A$ such that $\mathcal A \cap \mathfrak m < I < \mathcal A$.
Since $\mathfrak m$ is a maximal right ideal of $\Sigma$, hence 
$\mathcal I \cap S$ has an element $\beta$.  Let $0 \ne \alpha \in \mathcal A$, then by the right Ore property of 
$S$, 
$\alpha\beta' = \beta\alpha'$, where $\beta' \in S$ and $\alpha' \in \mathcal A$. 
It follows that $\mathcal A/ \mathcal I$ is $\mathcal B$-torsion.
Using \cite[Lemma 2.7]{BG} and Proposition \ref{BG_dim},
$\gk(\mathcal A/\mathcal I) < \rk(B)$.
Thus by Proposition \ref{GKdim_ft}, $S(\mathfrak m)$ is critical.
\end{proof}

\begin{theorem}
\label{dim_n-1_ctr}
If $\mathcal B$ is commutative, $\mathcal A$ has center $F$ and $\mathcal A \cap \mathfrak m$ contains an element unitary with respect to $\bar t$ then $S(\mathfrak m)$ is simple.
\end{theorem}

\begin{proof}
As $\mathcal A \cap \mathfrak m$ contains a unitary element, $S(\mathfrak m)$ is finitely generated as 
$\mathcal B$-module (see above).
As noted above, it suffices to show that $\Hom_{\mathcal A}(S(\mathfrak m), N) = 0$
for each $S$-torsion simple $\mathcal A$-module $N$. Indeed, if this was not the case then $N$ being an image of $S(\mathfrak m)$ and would also be finitely generated as $\mathcal B$-module.
But then by Proposition \ref{dim_n-1_cs}, $N$ is not $\mathcal B$-torsion.
The contradiction implies that $S(\mathfrak m)$ is simple.
\end{proof}

The following result was first obtained in \cite{AR} in a more general setting.
\begin{theorem}
If $\dim(\mathcal A) = 1$ and $\mathcal A \cap \mathfrak m$ contains an element unitary with respect to $\bar t$ then 
$S(\mathfrak m)$ is simple.
\end{theorem}

\begin{proof}
We note that
$S(\mathfrak m)$ is finitely generated as $\mathcal B$-module.
If $S(\mathfrak m)$ is not simple, let $N$ be a quotient by 
a nonzero proper submodule.
By Proposition \ref{ctr_max_idl}, 
\[ \gk(N) < \gk(S(\mathfrak m)) = \rk(B). \] 
Hence $\dim(\mathcal A) > 1$ by Corollary \ref{B_ine} contrary to the hypothesis.
\end{proof}

\section{Conclusion}
Let $\mathscr Z(\mathcal A) = F$. If $\dim(\mathcal A) = 1$ or $n -1$ then the modules induced from simple $\mathcal B$-modules remain simple. Similarly the module $S(\mathfrak m)$, where $\mathfrak m$ is a maximal right ideal of 
the PID $\Sigma = AS^{-1}$, is simple if 
$\mathfrak m \cap \mathcal A$ has a unitary element. For the hereditary case
($\dim(\mathcal A) = 1$): \[ \gk(V \otimes_{\mathcal B} \mathcal A) = n - 1 = \gk(S(\mathfrak m)), \] 
since $\dim(\mathcal A) =1$ and Corollary \ref{B_ine} imply $\gk(V) \ge n - 2$ for a simple $\mathcal B$-module $V$ .

In the case $\gk(\mathcal A) = n - 1$, $\dim_F(V) < \infty$ for any simple $\mathcal B$-module $V$ by the Hilbert Nullstellensatz. 
Hence
\[ \gk(V \otimes_{\mathcal B} \mathcal A) = 1, \  \   \   \   \  
  \gk(S(\mathfrak m)) = n - 1. \]

We conclude by conjecturing that for $\dim(\mathcal A) = n- 1$, $1$ and $n - 1$ are the only possible GK dimensions for simple $\mathcal A$-modules.

\end{document}